\documentclass[letterpaper, 10 pt, conference]{ieeeconf}  
\IEEEoverridecommandlockouts                              
\usepackage{cite}
\usepackage{amsmath,amssymb,amsfonts}
\usepackage{amsthm}
\usepackage{graphicx}
\usepackage{textcomp}
\usepackage{booktabs}
\usepackage{array}

\DeclareRobustCommand{\gd}{\dot{\vphantom{t}g}}

\theoremstyle{plain}
\newtheorem{theorem}{Theorem}
\newtheorem{lemma}{Lemma}
\newtheorem{proposition}{Proposition}
\newtheorem{corollary}{Corollary}
\theoremstyle{definition}
\newtheorem{assumption}{Assumption}
\newtheorem{definition}{Definition}
\theoremstyle{remark}

\AtBeginDocument{%
  \setlength{\abovedisplayskip}{2pt plus 1pt minus 1pt}%
  \setlength{\belowdisplayskip}{2pt plus 1pt minus 1pt}%
  \setlength{\abovedisplayshortskip}{1pt plus 1pt}%
  \setlength{\belowdisplayshortskip}{1pt plus 1pt}}
\makeatletter
\def\thm@space@setup{\thm@preskip=1pt plus 1pt minus 0.5pt \thm@postskip=\thm@preskip}
\renewenvironment{proof}[1][\proofname]{\par\pushQED{\qed}\normalfont\topsep2\p@\@plus1\p@\relax
  \trivlist\item[\hskip\labelsep\itshape#1\@addpunct{.}]\ignorespaces}{\popQED\endtrivlist\@endpefalse}
\def\section{\@startsection{section}{1}{\z@}{1.2ex plus 1ex minus 0.3ex}%
{0.6ex plus .5ex minus 0ex}{\normalfont\normalsize\centering\scshape}}%
\def\subsection{\@startsection{subsection}{2}{\z@}{0.8ex plus 1ex minus 0.3ex}%
{0.4ex plus .3ex minus 0ex}{\normalfont\normalsize\itshape}}%
\AtBeginDocument{\if@filesw\immediate\write\@auxout{\string\bibstyle{plain}}%
  \immediate\write\@auxout{\string\bibdata{refs}}\fi}
\makeatother

\title{\LARGE \bf
On Exact Critical-Region Boundaries in Continuous-Time\\ Multiparametric Optimal Control Problems
}

\author{Lida Lamakani and Efstratios N. Pistikopoulos
\thanks{This work was supported by the Texas A\&M Energy Institute, with partial support from the National Science Foundation through the GOALI (CBET-2401564) and RETRO (CBET-2312457) awards.}
\thanks{The authors are with the Texas A\&M Energy Institute and the Artie McFerrin Department of Chemical Engineering, Texas A\&M University, College Station, TX 77843, USA
        {\tt\small lida.lamakani@tamu.edu, stratos@tamu.edu}}%
}

\begin{document}
\maketitle
\thispagestyle{empty}
\pagestyle{empty}
\raggedbottom
\suppressfloats[t]

\begin{abstract}
Continuous-time multiparametric model predictive control solves the optimal control problem once, offline, for every initial state, and divides the initial-state space into critical regions, each with its own control law and sequence of active constraints. A boundary between two regions is either a hyperplane, given by one linear equation, or curved. The curved boundaries have not been computed to date; a classifier has been trained instead to distinguish the regions, so the partition has been known only approximately. In this paper we show that a curved boundary can be computed exactly, without solving any optimal control problem. The method rests on one property. Once the sequence of active constraints is fixed, together with the instants inside it at which a constraint becomes active or inactive (the switching times) and the instant at which the change of the neighbouring region first appears (the event time), every condition on a boundary is a linear equation in the initial state, so the computation of a boundary, curved or not, is exact. We fix one of these instants, a switching time or the event time, and solve the linear equations that correspond to it for the corresponding initial state on the boundary; stepping the fixed instant along the horizon yields the whole boundary. A hyperplane is the case in which no instant is left to fix. We illustrate the method on the Newell--Lee evaporator, which has seven critical regions and eight boundaries, four hyperplanes and four curved, computing one boundary of each type in full with the adjacent regions at its ends.
\end{abstract}


\section{Introduction}

Multiparametric model predictive control computes the optimal control law offline, as an explicit function of the initial state, and divides the initial-state space into critical regions with one control law each \cite{bemporad2002,pistikopoulos2020}. In discrete time the regions are polytopes whose number grows rapidly with the number of sampling intervals \cite{bemporad2002,pistikopoulos2020,biegler2007}; the continuous-time formulation avoids this growth \cite{sakizlis2005,sun2016,lamakani2026a,lamakani2026c}. There, by Pontryagin's minimum principle (PMP), the optimal input follows a sequence of arcs \cite{srinivasan2003}, and on each arc the set of active constraints does not change. A critical region is the set of initial states whose solution follows the same sequence of arcs. We call this sequence the \emph{label} of the region and write it as the list of the active constraints, so $\{3\}\{\}$ means two arcs, constraint~3 active on the first and no constraint on the second. The instants at which one arc ends and the next begins are the \emph{switching times}, an entry time where a constraint becomes active and an exit time where it is released. Inside a region they vary with the initial state; only the label is fixed.

In continuous time the regions need not be polytopes. A boundary is where one arc appears or vanishes. The switching times enter this condition nonlinearly, so the boundary is in general curved. For this reason the boundaries were not computed in \cite{sun2016,sun2019}; a classifier was trained instead to distinguish the regions. We showed in \cite{lamakani2026b} that with one input bound and a single switch every boundary is a hyperplane, and in \cite{lamakani2026d} that the type of a boundary, hyperplane or curved, can be read from the labels of its two neighbouring regions before any computation.
\suppressfloats[t]

\begin{figure*}[t]
\centering
\includegraphics[width=0.9\textwidth]{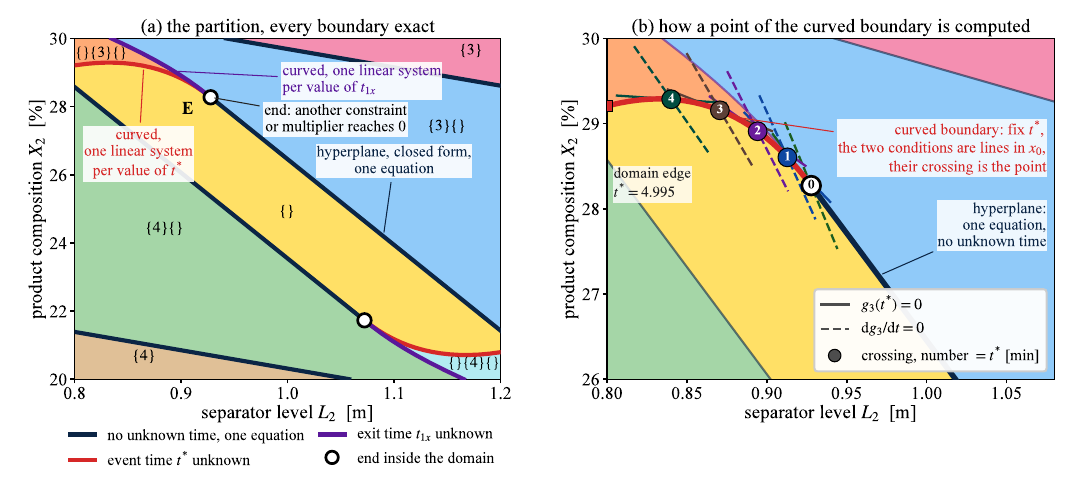}
\caption{(a) The seven critical regions of the evaporator and the eight boundaries, four hyperplanes, in closed form, and four curved, each by one linear system per value of the exit time $t_{1x}$, the instant at which constraint~3 is released, or of the event time $t^{*}$, the instant at which constraint~3 touches zero. (b) For each fixed $t^{*}$ the conditions $g_{3}(t^{*})=0$ (solid) and $\gd_{3}(t^{*})=0$ (dashed) are two lines, whose crossing is one exact point of the boundary. Both boundaries meet at the white point E, $t^{*}=0$.}
\label{fig:map}
\end{figure*}

The goal of this paper is to compute one boundary exactly, hyperplane or curved, to determine its two end points, and to identify the neighbouring regions at each end, without solving an optimal control problem. The key idea is how a boundary is computed. The initial state remains the parameter, and the regions and their boundaries are sets in the space of initial states, as in Fig.~\ref{fig:map}a. In the traditional multiparametric approach a boundary is obtained by solving equations in the initial state. In continuous time the equations of a boundary involve both the initial state and the switching times, and one of the two has to be fixed to solve for the other. They are linear in the initial state and nonlinear in the switching times (Lemma~\ref{lem:affine}). Fixing the initial state, as in \cite{sakizlis2005}, leaves nonlinear equations in the switching times, solved by iteration; a boundary is then located by searching over initial states and is known only approximately. In this paper we fix the switching time or event time instead; the equations that remain are linear in the initial state, and their solution, by linear algebra alone, is the exact point of the boundary at that time. Stepping the time along the horizon traces the whole curved boundary (Fig.~\ref{fig:map}b), so a curved boundary costs a sequence of small linear systems and no optimal control problem. For a hyperplane no time has to be fixed, and the one linear equation is the boundary itself.

Section~\ref{sec:walk} computes a boundary in five steps, each carried out on the evaporator of Fig.~\ref{fig:map} and then stated in general. Fig.~\ref{fig:steps} places these steps, boxes C and D, in the procedure that builds the partition \cite{lamakani2026d}, which is left to a companion paper.

\begin{figure}[tb]
\centering
\includegraphics[width=\columnwidth]{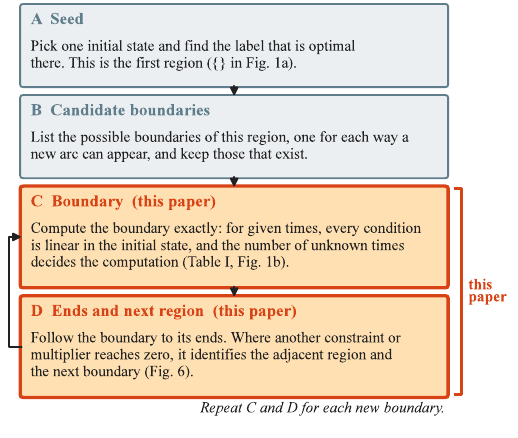}
\caption{The procedure that builds the partition \cite{lamakani2026d} and the part treated here.}
\label{fig:steps}
\end{figure}

\section{The Problem and the Example}\label{sec:problem}

\subsection{Optimal control problem}

The statements of this paper are written for a state of dimension $n=2$, the case of the example; the general case is treated in the companion paper. The state is $x(t)\in\mathbb{R}^{n}$ and the input $u(t)\in\mathbb{R}$ is scalar. The system is linear and time invariant,
\begin{equation}
\dot x(t)=Ax(t)+Bu(t),\qquad x(0)=x_{0},
\label{eq:dyn}
\end{equation}
and over the fixed horizon $[0,t_{f}]$ the cost to be minimised is
\begin{equation}
J(x_{0})=\tfrac12 x(t_{f})^{\!\top}\!Px(t_{f})
+\tfrac12\!\int_{0}^{t_{f}}\!\big(x^{\!\top}\!Qx+Ru^{2}\big)\,dt ,
\label{eq:cost}
\end{equation}
with $Q\succeq0$, $R>0$ and $P\succeq0$. Along the whole horizon the state and the input must respect $p$ path constraints,
\begin{equation}
g_{i}(x,u)=C_{x,i}\,x+C_{u,i}\,u-d_{i}\le 0,\qquad i=1,\dots,p,
\label{eq:con}
\end{equation}
and the initial state $x_{0}$ lies in a box $X\subset\mathbb{R}^{n}$, which we call the domain.
\begin{assumption}\label{as:mixed}
$C_{u,i}\neq 0$ for every $i$.
\end{assumption}
A bound on the input is the case $C_{x,i}=0$. Constraints on the state alone, with $C_{u,i}=0$, make the costate jump \cite{hartl1995} and are not treated here.

\subsection{The evaporator}\label{sec:model}

We use the forced-circulation evaporator of Newell and Lee \cite{newell1989} (Fig.~\ref{fig:evap}), as in \cite{lamakani2026d}. The states are the separator level $L_{2}$, in m, and the product composition $X_{2}$, in \% solids; the input is the product flow $F_{2}$, in kg/min, with the other inputs held at their nominal values. The nominal point is the steady state of \cite{newell1989}, $L_{2}=1.0$~m, $X_{2}=25$~\%, $F_{2}=2.0$~kg/min, and the variables are scaled deviations from it, $x_{1}=(L_{2}-1.0)/0.1$, $x_{2}=(X_{2}-25)/2.5$ and $u=(F_{2}-2.0)/0.5$. The linearised model is
\begin{equation}
A=\begin{bmatrix}0 & 0.1045\\ 0 & -0.1\end{bmatrix},\qquad
B=\begin{bmatrix}-0.25\\ -0.25\end{bmatrix},
\label{eq:AB}
\end{equation}
in $\mathrm{min}^{-1}$, rounded to four decimals; all numbers below use the unrounded model. We set $t_{f}=10$~min, $Q=I$, $R=1$, $P=I$ and the domain $X=[-2,2]^{2}$. The constraints are a band on the product flow, $|F_{2}-2.0|\le1.0$~kg/min (constraints 1 and 2), and a band on the solids leaving with the product, $43\le F_{2}X_{2}\le57$ (constraints 3 and 4), linearised at the nominal point, so that $C_{x}^{\!\top}=\left[\begin{smallmatrix}0&0&0&0\\0&0&5&-5\end{smallmatrix}\right]$, $C_{u}^{\!\top}=[\,1\;\;{-1}\;\;12.5\;\;{-12.5}\,]$ and $d^{\top}=[\,2\;\;2\;\;7\;\;7\,]$. The flow band is never reached in the domain, so only constraints 3 and 4 appear in the labels. The problem is symmetric under $x_{0}\mapsto-x_{0}$, which exchanges constraints 3 and 4, so every boundary has a mirror image, and we follow constraint~3 throughout.

\begin{figure}[tb]
\centering
\includegraphics[width=0.66\columnwidth]{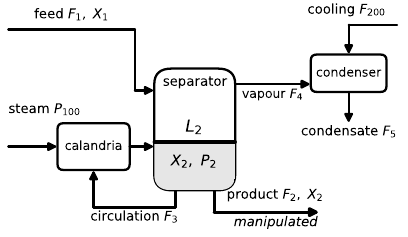}
\caption{The forced-circulation evaporator \cite{newell1989}.}
\label{fig:evap}
\end{figure}

\subsection{Optimality conditions}\label{sec:pmp}

Pontryagin's minimum principle expresses the optimal input through two further quantities. The costate $\lambda(t)\in\mathbb{R}^{n}$ carries the effect of the rest of the horizon into the input at time $t$; it satisfies $\dot\lambda=-\partial\mathcal{H}/\partial x$ with $\lambda(t_{f})=Px(t_{f})$ at the final time. The multiplier $\mu_{i}(t)\ge0$ of constraint $i$ is positive only while that constraint is active. Here $\mathcal{H}$ is the Hamiltonian,
\begin{equation}
\mathcal{H}=\tfrac12\big(x^{\!\top}\!Qx+Ru^{2}\big)+\lambda^{\!\top}(Ax+Bu)
+\textstyle\sum_{i}\mu_{i}\,g_{i}(x,u),
\end{equation}
and at every instant the optimal input minimises $\mathcal{H}$ over the admissible inputs; because $R>0$ the minimiser is unique and the conditions are necessary and sufficient \cite{abhijeet2024}. With no constraint active, $\partial\mathcal{H}/\partial u=0$ yields the unconstrained input $u=u_{\mathrm{free}}=-B^{\!\top}\!\lambda/R$. With constraint $i$ active, $g_{i}=0$ fixes $u=(d_{i}-C_{x,i}x)/C_{u,i}$, and $\partial\mathcal{H}/\partial u=0$ returns $\mu_{i}=-(Ru+B^{\!\top}\!\lambda)/C_{u,i}$. Since the input is scalar and an active constraint fixes it, at most one constraint is active on an arc. The solution with a given label is optimal at $x_{0}$ if its trajectory satisfies $g_{i}\le0$ for every inactive and $\mu_{i}\ge0$ for every active constraint over the whole horizon; we call this test the \emph{check}; the maxima are located from a closed-form bound on the second derivatives, not on a time grid. At the origin of the evaporator the free solution has $\lambda\equiv0$, $u\equiv0$ and $g=-d<0$, so $\{\}$ is optimal there; it is the first region of Fig.~\ref{fig:map}a, and we compute its boundaries with constraint~3.

\section{Computing a Boundary in Five Steps}\label{sec:walk}

Two regions meet at a boundary, and $S\,|\,S'$ denotes the boundary between the regions with labels $S$ and $S'$. We call $S$ the label with fewer arcs; the other label contains one additional arc, the \emph{extra arc}, inserted into $S$, which splits an arc of $S$ in two when it lies inside it. On the boundary the extra arc has zero length. A boundary is a set of initial states $x_{0}$, and we call them its \emph{points}.

\begin{definition}[switching times]\label{def:switch}
A switching time is an instant inside the solution of a label at which the active constraint changes, one arc ends and the next begins. It is an \emph{entry time} $t_{kt}$ if constraint $i$ becomes active there and an \emph{exit time} $t_{kx}$ if constraint $i$ is released there, $k$ numbering the constrained arcs of the label as in \cite{sakizlis2005}; a label with $K$ arcs has $K-1$ switching times, and they vary with $x_{0}$. The only free quantities in the solution of a label are the instants at which its arcs join. At an entry $g_{i}=0$, because a constraint starts to act exactly when it is reached; at an exit $\mu_{i}=0$, because it stops acting exactly when its multiplier, the force it exerts on the input, vanishes.
\end{definition}

\begin{definition}[event time]\label{def:event}
The event time $t^{*}$ is the instant at which the extra arc appears; on the boundary it has zero length, so it is a single instant. It can be the start of the horizon, $t^{*}=0$, its end, $t^{*}=t_{f}$, or an instant inside an arc of $S$. The \emph{event quantity} is $g_{j}$ if the extra arc is an arc of constraint $j$, and $\mu_{j}$ if it is a free arc adjacent to an arc of constraint $j$. The extra arc appears when the event quantity is zero at $t^{*}$. If $t^{*}$ lies inside an arc of $S$, the event quantity only touches zero there.
\end{definition}

The equations of a boundary are then obtained in three items, which steps 2 and 3 state and prove and Table~\ref{tab:bounds} applies to the evaporator.
\begin{enumerate}\setlength{\itemsep}{0pt}
\item[(i)] \emph{Switching conditions of $S$.} For each switching time of $S$ write $g_{i}=0$ at an entry and $\mu_{i}=0$ at an exit.
\item[(ii)] \emph{Event condition of the extra arc.} If the extra arc is an arc of constraint $j$, write $g_{j}(t^{*})=0$, and also $\gd_{j}(t^{*})=0$ when $t^{*}$ is inside a free arc, because $g_{j}$ then only touches zero. If the extra arc is a free arc, write $\mu_{j}(t^{*})=0$ for the constraint $j$ whose arc is adjacent to it, and also $\dot\mu_{j}(t^{*})=0$ when $t^{*}$ is inside that arc.
\item[(iii)] \emph{Unknown times and the fixed one.} Count the unknown times, $m$, the switching times of $S$ and $t^{*}$ only when it is inside an arc. If $m=0$ the boundary is a hyperplane; if $m=1$, fix that one instant, solve the two linear equations for $x_{0}$, and step it.
\end{enumerate}
Each step below first carries out the computation on the evaporator with its numbers, under the heading \emph{Evaporator}, and then states the general result that it illustrates.

\subsection{Step 1. The solution on the boundary as a formula in $x_{0}$}\label{sec:step1}

\subsubsection*{Evaporator} For the boundaries of $\{\}$, $S=\{\}$ and the whole horizon is one free arc, on which the state and the costate $z=(x,\lambda)$ follow $z(t)=e^{Mt}z(0)$ with $M$ from \eqref{eq:Mfree} and \eqref{eq:AB}. The terminal condition, \eqref{eq:term} with the blocks of $e^{Mt_{f}}$, yields
\begin{equation}
\lambda(0)=G\,x_{0},\qquad G=\begin{bmatrix}4.4107 & -0.5410\\ -0.5410 & 2.7870\end{bmatrix},
\label{eq:G}
\end{equation}
so $z(t;x_{0})=e^{Mt}Wx_{0}$. On a free arc $u=u_{\mathrm{free}}=0.25(\lambda_{1}+\lambda_{2})$, so constraint~3 is $g_{3}=5x_{2}+12.5u-7=c^{\!\top}z-7$ with $c^{\!\top}=[\,0\;\;5\;\;3.125\;\;3.125\,]$, and at any instant $t$
\begin{equation}
g_{3}(t)=\underbrace{c^{\!\top}e^{Mt}W}_{a(t)^{\!\top}\!,\ \text{nonlinear in }t}\;x_{0}-7 .
\label{eq:c3}
\end{equation}

\subsubsection*{What is fixed, and why} In \eqref{eq:c3}, $x_{0}$ enters linearly and $t$ only through the coefficient $a(t)$, which is nonlinear in $t$. For a given $x_{0}$, $g_{3}(t)=0$ is a nonlinear equation in $t$, and its root, a switching time, varies with $x_{0}$ (Definition~\ref{def:switch}). For a given $t$, $g_{3}(t)=0$ is one linear equation in $x_{0}$, at $t=2.5$~min the line $2.6230\,x_{0,1}+6.2224\,x_{0,2}=7$. By Lemma~\ref{lem:affine} every constraint function, multiplier and derivative has this form, $f(t)=a(t)^{\!\top}x_{0}+a_{0}(t)$, for every label (Fig.~\ref{fig:fact}); therefore the times are fixed and $x_{0}$ is solved for.

At every point of the boundary the extra arc has zero length, so the solution there consists of the arcs of $S$ only. Within one arc the set of active constraints does not change, so the state and the costate $z=(x,\lambda)$ obey a linear differential equation with constant coefficients,
\begin{equation}
\dot z = M_{k}z+b_{k},
\label{eq:arc}
\end{equation}
obtained by inserting the input of Section~\ref{sec:pmp} into \eqref{eq:dyn} and into $\dot\lambda=-\partial\mathcal{H}/\partial x$. For a free arc,
\begin{equation}
M_{k}=\begin{bmatrix}A & -BB^{\!\top}\!/R\\[2pt] -Q & -A^{\!\top}\end{bmatrix},
\qquad b_{k}=0 ,
\label{eq:Mfree}
\end{equation}
and, for the arc of constraint $i$, with $A_{i}=A-BC_{x,i}/C_{u,i}$ and $\Omega_{i}=Q+(R/C_{u,i}^{2})\,C_{x,i}^{\!\top}C_{x,i}$,
\begin{equation}
M_{k}=\begin{bmatrix}A_{i} & 0\\[2pt] -\Omega_{i} & -A_{i}^{\!\top}\end{bmatrix},
\qquad
b_{k}=\begin{bmatrix} B\,d_{i}/C_{u,i}\\[2pt]
(R\,d_{i}/C_{u,i}^{2})\,C_{x,i}^{\!\top}\end{bmatrix}.
\label{eq:Mb}
\end{equation}

\begin{lemma}\label{lem:affine}
Fix the label, with $K$ arcs, and its switching times $t_{1}<\dots<t_{K-1}$, entry or exit times, and write $h_{k}=t_{k}-t_{k-1}$ for the lengths of the arcs, with $t_{0}=0$ and $t_{K}=t_{f}$. On arc $k$ the solution of \eqref{eq:arc} is $z(t_{k})=\Phi_{k}z(t_{k-1})+\varphi_{k}$ with
\begin{equation}
\begin{bmatrix}\Phi_{k}&\varphi_{k}\\0&1\end{bmatrix}
=\exp\!\left(\!\begin{bmatrix}M_{k}&b_{k}\\0&0\end{bmatrix}\!h_{k}\right),
\label{eq:arcflow}
\end{equation}
and at every $t\in[0,t_{f}]$ the state and the costate are
\begin{equation}
z(t;x_{0})=\Phi(t)\big(Wx_{0}+w\big)+\varphi(t),
\label{eq:affine}
\end{equation}
where $\Phi(t)$ and $\varphi(t)$ are the products of the arc matrices up to $t$, $W=\left[\begin{smallmatrix}I\\G\end{smallmatrix}\right]$, $w=\left[\begin{smallmatrix}0\\\eta\end{smallmatrix}\right]$, $\lambda(0)=G\,x_{0}+\eta$ and, from the state and costate blocks of the flow over the whole horizon,
\begin{equation}
\begin{split}
G&=L^{-1}(P\Phi_{xx}-\Phi_{\lambda x}),\qquad
\eta=L^{-1}(P\varphi_{x}-\varphi_{\lambda}),\\
L&=\Phi_{\lambda\lambda}-P\Phi_{x\lambda}.
\end{split}
\label{eq:term}
\end{equation}
The coefficients depend on the label, the switching times and $t$ but not on $x_{0}$, so for fixed switching times the trajectory is affine in $x_{0}$.
\end{lemma}
\begin{proof}
See \cite{lamakani2026d}; \eqref{eq:term} is the terminal condition $\lambda(t_{f})=Px(t_{f})$ solved for $\lambda(0)$, and $L$ is invertible because the problem with fixed arcs is strictly convex.
\end{proof}

\begin{figure}[tb]
\centering
\includegraphics[width=0.92\columnwidth]{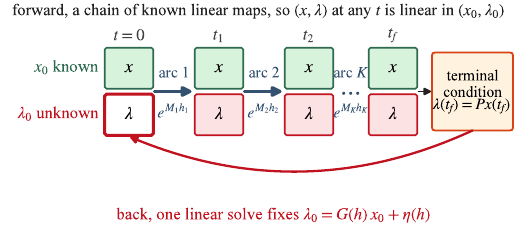}
\caption{Lemma~\ref{lem:affine} in a picture. Forward, $(x,\lambda)$ propagates through the arcs, so it is linear in $(x_{0},\lambda_{0})$ at every instant; back, the terminal condition fixes $\lambda_{0}$ by one linear solve.}
\label{fig:fact}
\end{figure}

\subsection{Step 2. The equations that hold on the boundary}\label{sec:step2}

\subsubsection*{Evaporator} Fig.~\ref{fig:horizon} shows the conditions along the horizon and Table~\ref{tab:bounds} lists them for the four boundaries of constraint~3. Consider row 3, the boundary between $\{3\}\{\}$ and $\{\}\{3\}\{\}$, with $S=\{3\}\{\}$. In $\{3\}\{\}$ constraint~3 is active on $[0,t_{1x}]$ and exits at $t_{1x}$, so $\mu_{3}(t_{1x})=0$ holds at every $x_{0}$ of $\{3\}\{\}$, the switching condition (i). This equation determines $t_{1x}$ for a given $x_{0}$, at the point B $=(-0.45,\,1.40)$ of Fig.~\ref{fig:horizon} $t_{1x}=5.44$~min, and it does not involve the boundary. The boundary adds the extra arc of $\{\}\{3\}\{\}$, a free arc at the start of the horizon, so the multiplier of constraint~3 vanishes at $t=0$, $\mu_{3}(0)=0$ with $t^{*}=0$, the event condition (ii). Inside $\{\}\{3\}\{\}$ itself, constraint~3 enters at $t_{1t}$ and exits at $t_{1x}$, $g_{3}(t_{1t})=0$ and $\mu_{3}(t_{1x})=0$; at the point C $=(-1.30,\,1.72)$ these yield $t_{1t}=1.38$~min and $t_{1x}=4.55$~min.

\begin{figure}[tb]
\centering
\includegraphics[width=\columnwidth]{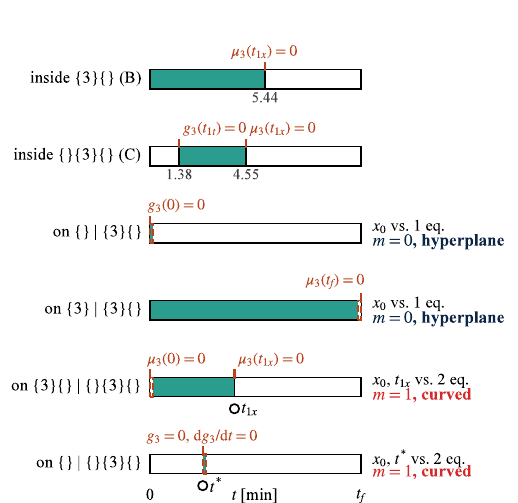}
\caption{The horizon at six initial states, B, C and one on each of the four boundaries of constraint~3. Teal, constraint~3 active; white, free; dashed, the extra arc. Red ticks, the equations of Proposition~\ref{prop:switch}; open circle, the fixed switching time; $m$ counts the unknowns besides $x_{0}$.}
\label{fig:horizon}
\end{figure}

\begin{table*}[t]
\caption{The equations of the four boundaries of constraint~3 of the evaporator, two states ($n=2$), steps 2 and 3; $m$ counts the unknowns besides $x_{0}$.}
\label{tab:bounds}
\centering
\footnotesize\renewcommand{\arraystretch}{1.0}\setlength{\tabcolsep}{4.2pt}
\begin{tabular}{@{}l l l l l l l c@{}}
\toprule
 & & \multicolumn{2}{c}{switching time of $S$ (holds in all of $S$)} & \multicolumn{2}{c}{event time $t^{*}$ (holds on the boundary only)} & & \\
\cmidrule(lr){3-4}\cmidrule(lr){5-6}
boundary & $S$ & description & switching condition (i) & description & event condition (ii) & fixed & $m$ \\
\midrule
$\{\}\,|\,\{3\}\{\}$ & $\{\}$ & no switching time & -- & constraint~3 enters at $t^{*}=0$ & $g_{3}(0)=0$ & -- & 0 \\
$\{3\}\,|\,\{3\}\{\}$ & $\{3\}$ & no switching time & -- & constraint~3 exits at $t^{*}=t_{f}$ & $\mu_{3}(t_{f})=0$ & -- & 0 \\
$\{3\}\{\}\,|\,\{\}\{3\}\{\}$ & $\{3\}\{\}$ & constraint~3 exits at $t_{1x}$ & $\mu_{3}(t_{1x})=0$ & free arc at $t^{*}=0$ & $\mu_{3}(0)=0$ & $t_{1x}$ & 1 \\
$\{\}\,|\,\{\}\{3\}\{\}$ & $\{\}$ & no switching time & -- & constraint~3 touches zero at $t^{*}$ & $g_{3}(t^{*})=0$, $\gd_{3}(t^{*})=0$ & $t^{*}$ & 1 \\
\bottomrule
\end{tabular}
\end{table*}

\begin{proposition}[switching conditions]\label{prop:switch}
Let the active constraint change at an instant $\tau$, inside a region or on its boundary.
\begin{enumerate}\setlength{\itemsep}{1pt}
\item[(a)] If $\tau$ is an entry of constraint $i$ from a free arc, then $g_{i}(\tau)=0$.
\item[(b)] If $\tau$ is an exit of constraint $i$ into a free arc, then $\mu_{i}(\tau)=0$.
\item[(c)] If entry and exit coincide at $\tau$, the quantity that reaches zero only touches it, so its value and its time derivative both vanish. For an arc of constraint $i$ inside a free arc, $g_{i}(\tau)=0$ and $\gd_{i}(\tau)=0$; for a free arc inside the arc of constraint $i$, $\mu_{i}(\tau)=0$ and $\dot\mu_{i}(\tau)=0$.
\end{enumerate}
\end{proposition}
These are the junction conditions of the minimum principle \cite{hartl1995}; on a free arc $g_{i}<0$ and on the arc of constraint $i$ $\mu_{i}>0$, so each reaches zero where the arcs join, and a quantity that only touches zero has zero derivative there. On a boundary Proposition~\ref{prop:switch} applies in two places. At each switching time of $S$ it yields the switching condition (i), and at the event time $t^{*}$ the event condition (ii) of the extra arc. A direct switch from one constraint to another, as in $\{3\}\{1\}\{\}$, has its own condition and is left to the companion paper; the evaporator has none.

\subsection{Step 3. Count the unknowns, fix one and solve}\label{sec:step3}\label{sec:step4}

\subsubsection*{Evaporator} The last two columns of Table~\ref{tab:bounds} show what is fixed. On the hyperplanes $\{\}\,|\,\{3\}\{\}$ and $\{3\}\,|\,\{3\}\{\}$ nothing is fixed, $m=0$; on $\{3\}\{\}\,|\,\{\}\{3\}\{\}$ the exit time $t_{1x}$ is fixed and on $\{\}\,|\,\{\}\{3\}\{\}$ the event time $t^{*}$, $m=1$. Row 1 of Table~\ref{tab:bounds}: with $z(0)=Wx_{0}$, $g_{3}(0)=c^{\!\top}Wx_{0}-7=0$ is
\begin{equation}
12.0927\,x_{0,1}+12.0188\,x_{0,2}=7,
\label{eq:line}
\end{equation}
or $25.15\,L_{2}+X_{2}=51.61$ in the units of Fig.~\ref{fig:map}a, its blue line. Row 2: $\mu_{3}(t_{f})=0$ on the one-arc label $\{3\}$ yields the line $5.38\,L_{2}+X_{2}=35.07$, in agreement with \cite{lamakani2026d}. Row 3: on the trajectory of $\{3\}\{\}$ with its exit time fixed at $t_{1x}=4$~min (Lemma~\ref{lem:affine}), the two equations $\mu_{3}(0)=0$ and $\mu_{3}(t_{1x})=0$ yield the point $(-1.0835,\,1.6288)$; the check is satisfied, $\mu_{3}\ge0$ up to $t_{1x}$ and $g_{3}\le0$ after it, the fifth row of Fig.~\ref{fig:horizon}; stepping $t_{1x}$ traces the purple curve. Its event time is $0$ at every point, yet the boundary is curved, because $t_{1x}$ enters the costate at $t=0$ through the terminal condition. Row 4: with $z(t^{*})=e^{Mt^{*}}Wx_{0}$ the two equations are $c^{\!\top}e^{Mt^{*}}W\,x_{0}=7$ and $c^{\!\top}Me^{Mt^{*}}W\,x_{0}=0$, that is $N(t^{*})\,x_{0}=r$ with $r=(7,0)$, so the boundary is the explicit curve $x_{0}(t^{*})=N(t^{*})^{-1}r$; at $t^{*}=2.5$~min,
\begin{equation}
\begin{bmatrix}2.6230 & 6.2224\\ -2.1852 & -1.5784\end{bmatrix}x_{0}=\begin{bmatrix}7\\0\end{bmatrix}
\;\Rightarrow\;
x_{0}=\begin{bmatrix}-1.1683\\ \phantom{-}1.6174\end{bmatrix},
\label{eq:pt25}
\end{equation}
and the check is satisfied at this point, since along the free trajectory from it $g_{1},\dots,g_{4}\le0$ over the whole horizon and $g_{3}$ touches zero only at $t^{*}=2.5$~min, the last row of Fig.~\ref{fig:horizon}. Fig.~\ref{fig:map}b shows the crossings for $t^{*}=0,\dots,4$, each checked in the same way; stepping $t^{*}$ traces the red curve of Fig.~\ref{fig:map}a.

\begin{theorem}[the unknowns of a boundary]\label{thm:boundary}
Let $S$ have $K$ arcs. The unknowns of the boundary are $x_{0}$, the $K-1$ switching times of $S$ and, when the extra arc is inside an arc of $S$, the event time $t^{*}$; when the extra arc is at $0$ or at $t_{f}$, $t^{*}$ is known. The number of unknown times is
\begin{equation}
m=(K-1)+\begin{cases}1 & \text{if the extra arc is inside an arc of } S,\\ 0 & \text{otherwise,}\end{cases}
\label{eq:m}
\end{equation}
and the number of equations is $m+1$, the $K-1$ switching conditions and the event condition. Hence, when the equations are independent, the boundary has dimension $n-1$ for every $K$ and every position of the extra arc, and for fixed switching times every equation is linear in $x_{0}$. For $m=1$ the time that is fixed is $t^{*}$ when $S$ has one arc and the switching time of $S$ when it has two.
\end{theorem}
\begin{proof}
The equations are Proposition~\ref{prop:switch} at each instant where the active constraint changes on the boundary; an extra arc inside an arc adds one unknown, $t^{*}$, and one equation, the slope. There are $n+m$ unknowns and $m+1$ equations, and linearity is Lemma~\ref{lem:affine}.
\end{proof}

\begin{corollary}[computation of a boundary]\label{cor:compute}
Table~\ref{tab:cases} summarises the computation. With $m=0$ the one equation is the boundary. With $m=1$, fix the switching time of Theorem~\ref{thm:boundary}, solve the two linear equations for $x_{0}$, the point of the boundary at which the switch takes place at that instant, and step it over its range. The evaporator has no boundary with $m\ge2$. Each computed $x_{0}$ is then tested by the check of Section~\ref{sec:pmp} over the whole horizon, with the event quantity reaching zero only at $t^{*}$; a point that fails lies past an end of the boundary (step 4).
\end{corollary}

\begin{table}[tb]
\caption{Computation of a boundary for each $m$ (Corollary~\ref{cor:compute}).}
\label{tab:cases}
\centering
\footnotesize\renewcommand{\arraystretch}{1.0}\setlength{\tabcolsep}{2pt}
\begin{tabular}{@{}c >{\raggedright\arraybackslash}p{2.25cm} l l@{}}
\toprule
$m$ & fixed & determined & method \\
\midrule
$0$ & none & $x_{0}$ & one linear equation, closed form \\
$1$ & $t^{*}$ or the switching time of $S$ & $x_{0}$ & two linear equations in $x_{0}$ \\
$\ge2$ & one of them & the others, $x_{0}$ & Newton, then linear equations \\
\bottomrule
\end{tabular}
\end{table}

\subsection{Step 4. The ends of the boundary}\label{sec:step5}

\subsubsection*{Evaporator} The hyperplane boundary given by \eqref{eq:line} reaches the side $x_{0,1}=2$ of the domain, case (a) of Table~\ref{tab:ends}, at $(2.0000,\,-1.4299)$. At its other end every point has $g_{3}(0)=0$, and $\{\}$ remains optimal only while $g_{3}$ decreases immediately after $t=0$, so the boundary ends where the slope $\gd_{3}(0)=c^{\!\top}MWx_{0}=s^{\!\top}x_{0}$, $s^{\!\top}=[\,-5.9442\;\;-3.2792\,]$, is zero, case (c); solving \eqref{eq:line} with $s^{\!\top}x_{0}=0$ yields $x_{0}^{E}=(-0.7221,\,1.3090)$, the point E of Fig.~\ref{fig:map}a. The curve of row 4 starts at E, where $t^{*}=0$, case (d), and ends on the side $x_{0,1}=-2$ at $t^{*}=4.995$~min, case (a). The curve of row 3 starts at E in the limit $t_{1x}\to0$, case (d), and ends on the side $x_{0,2}=2$ at $t_{1x}=7.74$~min, case (a).

\begin{proposition}[the ends of a boundary]\label{prop:ends}
Follow the boundary, along the hyperplane if $m=0$ or by stepping the fixed switching time if $m=1$. The boundary ends at the first point where one more equation holds, one of the four of Table~\ref{tab:ends}. For $m=0$ that equation and the boundary equation yield the end point; for $m=1$ it determines the fixed switching time, and $x_{0}$ follows from the two equations of the boundary. Which of the four occurs first is identified by the check at the computed points.
\end{proposition}

\begin{table*}[t]
\caption{The four ways a boundary ends, cases (a) to (d) of Proposition~\ref{prop:ends}, each adding one equation. The extra arc is an arc of constraint $j$, constraint~3 in the evaporator, and $k$ is any other constraint.}
\label{tab:ends}
\centering
\footnotesize\renewcommand{\arraystretch}{1.0}\setlength{\tabcolsep}{4pt}
\begin{tabular}{@{}l >{\raggedright\arraybackslash}p{3.6cm} >{\raggedright\arraybackslash}p{2.4cm} >{\raggedright\arraybackslash}p{5.0cm} >{\raggedright\arraybackslash}p{4.9cm}@{}}
\toprule
case & the boundary ends where & can occur & added equation & after the end \\
\midrule
(a) & it reaches a side of the domain & on every boundary & $x_{0,1}$ or $x_{0,2}$ at its bound & the boundary exits the domain \\
(b) & another constraint $k$ becomes active or inactive & on every boundary & $g_{k}=0$ (new entry or touch) or $\mu_{k}=0$ (new exit) & a new arc of constraint $k$; four regions meet (Prop.~\ref{prop:regions}) \\
(c) & the extra arc detaches from the end of the horizon & only if $t^{*}=0$ or $t_{f}$ & $\gd_{j}=0$ or $\dot\mu_{j}=0$ at that end & the extra arc lies inside the horizon; three regions meet (Prop.~\ref{prop:regions}) \\
(d) & the fixed switching time reaches the end of its range & only if $m\ge1$ & a switching time of $S$ equal to $0$, $t_{f}$ or its neighbour, or $t^{*}=0$ or $t_{f}$ & an arc of $S$ or the extra arc vanishes \\
\bottomrule
\end{tabular}
\end{table*}
\begin{proof}
Past an end, $S$ is no longer the solution on one side of the boundary. By Proposition~\ref{prop:switch} the solution changes only where a constraint becomes active or inactive, so one of three events occurs at the end: another constraint $k$ becomes active or inactive, case (b) of Table~\ref{tab:ends}; the constraint $j$ of the extra arc would be violated next to the end of the horizon, case (c); or an arc of $S$ or the extra arc vanishes, case (d). Otherwise the boundary reaches a side of the domain, case (a). Each case adds one equation, so by the count of Theorem~\ref{thm:boundary} the end is a point.
\end{proof}

\subsection{Step 5. The regions at an end}\label{sec:step6}

\begin{figure}[tb]
\centering
\includegraphics[width=0.75\columnwidth]{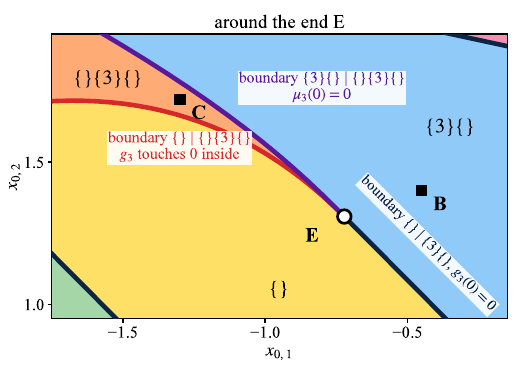}
\caption{The end E of the hyperplane boundary $\{\}\,|\,\{3\}\{\}$ (blue), where the curved boundaries $\{\}\,|\,\{\}\{3\}\{\}$ (red) and $\{3\}\{\}\,|\,\{\}\{3\}\{\}$ (purple) begin.}
\label{fig:end}
\end{figure}

\subsubsection*{Evaporator} Consider the free solution, label $\{\}$, from an initial state near E, and the signs of $g_{3}(0)$ and of the slope $\gd_{3}(0)$ along it. If $g_{3}(0)<0$ and $\gd_{3}(0)<0$, constraint~3 is satisfied at $t=0$ and $g_{3}$ decreases, so the free solution is optimal and $x_{0}$ lies in $\{\}$. If $g_{3}(0)>0$, the free solution violates constraint~3 at $t=0$, so constraint~3 is active from the start and $x_{0}$ lies in $\{3\}\{\}$, whatever the sign of $\gd_{3}(0)$. If $g_{3}(0)<0$ but $\gd_{3}(0)>0$, $g_{3}$ increases from a negative value and reaches zero inside the horizon, so constraint~3 becomes active on an inner arc and $x_{0}$ lies in $\{\}\{3\}\{\}$. At E both $g_{3}(0)$ and $\gd_{3}(0)$ are zero, so the three cases, and the three regions, meet there (Fig.~\ref{fig:end}). The boundaries that start at E are rows 3 and 4 of Table~\ref{tab:bounds}; the remaining boundaries of Fig.~\ref{fig:map}a are row 2 and the mirror images.

\begin{proposition}[the regions at an end]\label{prop:regions}
Let $f_{1}=0$ be the event condition of the boundary being followed and $f_{2}$ the quantity that reaches zero at its end $E$ inside the domain, case (b) or (c) of Table~\ref{tab:ends}. Both are taken negative inside $S$, $f=g$ for a constraint function and $f=-\mu$ for a multiplier. In case (b), $f_{2}=0$ is another boundary of $S$ starting at $E$, and four regions are adjacent to $E$, namely $S$, $S$ with the arc of $f_{1}$, $S$ with the arc of $f_{2}$, and $S$ with both arcs. In case (c), $f_{1}$ and $f_{2}$ are the value and the slope of one constraint at an end of the horizon. Three regions are adjacent to $E$, namely $S$, $S$ with that constraint active from that end, and $S$ with it active inside the horizon, and the three boundaries between them meet at $E$.
\end{proposition}
\begin{proof}
Near $E$ the signs of $f_{1}$ and $f_{2}$ along the trajectory of $S$ divide the initial states into four parts, and by Proposition~\ref{prop:switch} a positive sign inserts the corresponding arc, four labels in case (b). In case (c), with $f_{1}=g(0)$ and $f_{2}=\gd(0)$, $g(0)>0$ makes the constraint active from $t=0$ for either sign of $\gd(0)$, and $g(0)<0$ with $\gd(0)>0$ makes it active inside the horizon, three labels.
\end{proof}

\section{Conclusion}

This paper has shown how to compute one boundary of a critical region exactly in continuous time, without solving an optimal control problem. Once the label and its switching times are fixed, the state and the costate at every instant are linear in the initial state (Lemma~\ref{lem:affine}), so every condition that holds on a boundary is a linear equation in the initial state. These conditions are the switching conditions of the label $S$ and the event condition of the extra arc (Proposition~\ref{prop:switch}), one more than the unknown switching times (Theorem~\ref{thm:boundary}). With no unknown switching time the boundary is a hyperplane, one linear equation. With one, we fix it, solve two linear equations for the corresponding initial state and step it along the horizon, so every point of the curved boundary is exact (Corollary~\ref{cor:compute}). The boundary ends where one more quantity reaches zero (Proposition~\ref{prop:ends}), and the signs of the two quantities at the end identify the adjacent regions (Proposition~\ref{prop:regions}). On the Newell--Lee evaporator this yields the four hyperplanes in closed form and the four curved boundaries by two-by-two linear systems. The procedure that visits every boundary, general state dimension and several unknown switching times are left to the companion paper.

\end{document}